\documentclass[11pt]{amsart}
\usepackage[utf8]{inputenc}
\usepackage{setspace}
\usepackage[margin=1in]{geometry}
\usepackage{cite, hyperref, csquotes, url}
\usepackage{tikz, graphicx, marvosym, subfigure, float}
\usetikzlibrary{decorations.pathreplacing}
\usepackage{dsfont, amsfonts, amssymb, amsmath, amsthm, enumitem}
\usepackage{algorithm}
\usepackage[noend]{algpseudocode}
\numberwithin{equation}{section}
\newtheorem{theorem}{Theorem}

\newtheorem{lemma}[theorem]{Lemma}

\newtheorem{claim}[theorem]{Claim}
\newtheorem{remark}[theorem]{Remark}

\renewcommand{\pmod}[1]{\ (\mathrm{mod}\ #1)}

\title{A note on the mutual-visibility coloring of hypercubes}

\author{Maria Axenovich}

\address{Maria Axenovich \newline Karlsruhe Institute of Technology, Englerstraße 2, D-76131 Karlsruhe, Germany}
\email{maria.aksenovich@kit.edu}

\author{Dingyuan Liu}

\address{Dingyuan Liu \newline Karlsruhe Institute of Technology, Englerstraße 2, D-76131 Karlsruhe, Germany}
\email{liu@mathe.berlin}

\begin{document}
\begin{abstract}
A subset $M$ of vertices in a graph $G$ is a mutual-visibility set if for any two vertices $u,v\in{M}$ there exists a shortest $u$-$v$ path in $G$ that contains no elements of $M$ as internal vertices. Let $\chi_{\mu}(G)$ be the least number of colors needed to color the vertices of $G$, so that each color class is a mutual-visibility set. Let $n\in\mathbb{N}$ and $Q_{n}$ be an $n$-dimensional hypercube. It was proved by the authors that the maximum size of a mutual-visibility set in $Q_{n}$ is at least $\Omega(2^{n})$. Klav\v{z}ar, Kuziak, Valenzuela-Tripodoro, and Yero further asked whether it is true that $\chi_{\mu}(Q_{n})=O(1)$. In this note we answer their question in the negative by showing that $$\omega(1)=\chi_{\mu}(Q_{n})=O(\log\log{n}).$$
\end{abstract}
\maketitle

\vspace{-1em}
\section{Introduction}
Let $G$ be a simple graph. A subset $M\subseteq{V(G)}$ of vertices is called a \textit{mutual-visibility set} if any two vertices $u,v\in{M}$ can ``see'' each other, that is, there exists a shortest $u$-$v$ path in $G$ that contains no elements of $M$ as internal vertices. As in all other extremal problems, we are interested in the largest size of a mutual-visibility set in a given graph $G$, denoted as $\mu(G)$. The systematic investigation was pioneered by Di Stefano\cite{di2022mutual} and has garnered extensive attention and subsequent research\cite[etc.]{axenovich2024visibility,boruzanli2024mutual,brevsar2024lower,cicerone2024cubelike,cicerone2023variety,cicerone2023mutual,cicerone2024mutual,cicerone2024diameter,korvze2024mutual,korvze2024variety} in recent years. As the mutual-visibility problem was initially motivated by establishing efficient and confidential communication in networks, the research on $\mu(G)$ mainly focuses on sparse and highly connected graphs, such as product graphs and hypercube-like graphs. For $n\in\mathbb{N}$, the \textit{$n$-dimensional hypercube} $Q_{n}$ is a graph on the vertex set $2^{[n]}$, such that two vertices $A,B\in2^{[n]}$ form an edge in $Q_{n}$ if and only if their \textit{symmetric difference} $A\Delta{B}:=(A\backslash{B})\cup(B\backslash{A})$ has size $1$. It is known that $Q_{n}$ contains a large mutual-visibility set.
\begin{theorem}[{\hspace{-0.3mm}\cite[Theorem 1.2]{axenovich2024visibility}}]
\label{old}
For every $n\in\mathbb{N}$, we have $\mu(Q_{n})>0.186\cdot2^{n}$.
\end{theorem}

Recently, Klav\v{z}ar, Kuziak, Valenzuela-Tripodoro, and Yero\cite{klavzar2024color} introduced the coloring version of the mutual-visibility problem. Given a coloring of the vertices of $G$, we say that $G$ is \textit{properly colored} if each color class is a mutual-visibility set in $G$. The function considered in their paper is $\chi_{\mu}(G)$, which is the least number of colors needed to properly color $V(G)$. Equivalently, $\chi_{\mu}(G)$ is the smallest integer such that $V(G)$ can be partitioned into $\chi_{\mu}(G)$ mutual-visibility sets. It is easy to see that $\chi_{\mu}(G)\geq\lvert{V(G)}\rvert/\mu(G)$. Naturally, one might ask whether $\chi_{\mu}(G)=O\left(\lvert{V(G)}\rvert/\mu(G)\right)$ could be true in general. Since $\mu(Q_{n})=\Omega(Q_{n})$, Klav\v{z}ar et al.\cite{klavzar2024color} raised the following question:
\begin{center}
\textit{Is there an absolute constant $C>0$, such that $\chi_{\mu}(Q_{n})\leq{C}$ holds for all $n\in\mathbb{N}$?}
\end{center}

We answer their question in the negative with the following result. In particular, this shows that $\chi_{\mu}(G)$ can be arbitrarily far from the trivial lower bound $\lvert{V(G)}\rvert/\mu(G)$.
\begin{theorem}
\label{new}
$\omega(1)=\chi_{\mu}(Q_{n})=O(\log\log{n})$.
\end{theorem}

\section{Proof of Theorem \ref{new}: lower bound}
\begin{proof}[Proof of the lower bound]
Fix any positive integer $q$. It suffices to show that there exists $n_{0}>0$, such that $\chi_{\mu}(Q_{n})>q$ holds for all $n\geq{n_{0}}$.

Let $n\geq{n_{0}}$ with some $n_{0}>0$ to be determined later and fix an arbitrary $q$-coloring of $V(Q_{n})$. We shall argue that $Q_{n}$ is not properly colored, i.e., some color class is not a mutual-visibility set in $Q_{n}$. To do this, we first prove the following claim. Let $n\geq{k}\in\mathbb{N}$. The \textit{$k$th layer} of a subgraph $Q\subseteq{Q_{n}}$, denoted by $\mathcal{L}_{k}(Q)$, is defined as $V(Q)\cap\binom{[n]}{k}$.
\begin{claim}
\label{pop}
Let $n\geq{n'}\geq2$ and $Q\subseteq{Q_{n}}$ be a copy of $Q_{n'}$. If $M\subseteq{V(Q_{n})}$ contains three layers of $Q$, then $M$ is not a mutual-visibility set in $Q_{n}$.
\end{claim}
\begin{proof}[Proof of Claim \ref{pop}]
Let $\mathcal{L}_{i}$, $\mathcal{L}_{j}$ and $\mathcal{L}_{k}$ be the three layers of $Q$ contained in $M$, where $i<j<k$. Take two vertices $A\in\mathcal{L}_{i}$ and $B\in\mathcal{L}_{k}$ with $A\subseteq{B}$. Observe that every shortest $A$-$B$ path goes through some vertex $C\in{V(Q_{n})}$ satisfying $\lvert{C}\rvert=j$ and $A\subseteq{C}\subseteq{B}$. As $Q$ is a copy of $Q_{n'}$, all such vertices $C$ are contained in $\mathcal{L}_{j}$. Namely, every shortest $A$-$B$ path must go through the layer $\mathcal{L}_{j}$, so $M$ is not a mutual-visibility set.
\end{proof}
To show that $Q_{n}$ is not properly colored, by Claim \ref{pop} it suffices to find a copy of lower dimensional hypercube, which has three layers receiving the same color. Our argument here is a generalization of that in \cite[Lemma 2]{axenovich2017boolean}. Given integers $s\geq{k}\geq2$, the \textit{$q$-color hypergraph Ramsey number} $r_{k}(s;q)$ is defined as the smallest integer $r$ such that any $q$-coloring of $\binom{[r]}{k}$ contains a monochromatic copy of $\binom{[s]}{k}$. Since $q$ is fixed in the beginning, for simplicity we write $r_{k}(s)=r_{k}(s;q)$. Now let
\begin{equation}
\label{tow}
n_{0}=q\cdot\left(r_{2}\circ{r_{3}}\circ\dots\circ{r_{2q}(2q)}\right).
\end{equation}
Recall that $V(Q_{n})=2^{[n]}$ and $n\geq{n_{0}}$. By the pigeonhole principle there exists $X_{1}\subseteq[n]$ with
\begin{equation}
\lvert{X_{1}}\rvert\geq{n/q}\geq{r_{2}\circ{r_{3}}\circ\dots\circ{r_{2q}(2q)}},
\end{equation}
such that $\binom{X_{1}}{1}$ is monochromatic. Then, due to the size of $X_{1}$, there exists $X_{2}\subseteq{X_{1}}$ with
\begin{equation}
\lvert{X_{2}}\rvert\geq{r_{3}}\circ\dots\circ{r_{2q}(2q)},
\end{equation}
such that $\binom{X_{2}}{2}$ is monochromatic. By repeating this argument we obtain a sequence of sets
\begin{equation}
X_{2q}\subseteq{X_{2q-1}}\subseteq\dots\subseteq{X_{2}}\subseteq{X_{1}}\subseteq[n],
\end{equation}
such that $\lvert{X_{2q}}\rvert=2q$ and $\binom{X_{k}}{k}$ is monochromatic for each $k\in[2q]$. Let $Q$ be the subgraph of $Q_{n}$ induced by $2^{X_{2q}}$, in particular, $Q$ is a copy of $Q_{2q}$. We have that every layer of $Q$ is monochromatic. Since $Q$ contains $2q+1$ layers and there are $q$ colors, by the pigeonhole principle at least three layers of $Q$ receive the same color. This completes the proof of the lower bound in Theorem \ref{new}.
\end{proof}

\begin{remark}
Although the proof above shows that $\chi_{\mu}(Q_{n})$ grows with $n$, the dependence of the lower bound on $n$ is quite unsatisfactory. In fact, it is even difficult to express the lower bound on $\chi_{\mu}(Q_{n})$ in terms of $n$, since for $\chi_{\mu}(Q_{n})>q$ we require $n$ to be at least a composition of $q$-color hypergraph Ramsey numbers (see \eqref{tow}), which is a tower function of $q$ with height roughly $\Theta(q^{2})$ (for references on bounds on hypergraph Ramsey numbers, see, e.g., \cite{mubayi2020ramsey}). It would be interesting to actually obtain an expressible lower bound.
\end{remark}

\section{Proof of Theorem \ref{new}: upper bound}
Our proof of the upper bound consists two steps. First, we will reduce the problem from coloring the entire vertex set $V(Q_{n})$ to coloring a single layer of $Q_{n}$ in a certain way. Then, we show that a random coloring of the layer is potentially a good coloring. To achieve the first step, we need the following lemma from \cite{axenovich2024visibility}.
\begin{lemma}[{\hspace{-0.3mm}\cite[Lemma 2.2]{axenovich2024visibility}}]
\label{ear}
Let $n,g\in\mathbb{N}$, $g\geq3$. For each $k\in[0,n]:=\{0,1,\dots,n\}$, let $\mathcal{F}_{k}\subseteq\binom{[n]}{k}$ be such that
\begin{equation}
\label{era}
\forall\,A,B\in2^{[n]}\text{ with }\lvert{A}\rvert+g\leq{k}\leq\rvert{B}\rvert-g,\,\exists\,T\in\binom{[n]}{k}\backslash\mathcal{F}_{k}:\,A\cap{B}\subseteq{T}\subseteq{A}\cup{B}.
\end{equation}
Let $\lambda\in[g]$ and $I_{\lambda}=\left\{k\in[0,n]:\,k\equiv\lambda\pmod{g}\right\}$. Then $M:=\bigcup_{k\in{I_{\lambda}}}\mathcal{F}_{k}$ is a mutual-visibility set in $Q_{n}$.
\end{lemma}

From Lemma \ref{ear} we derive the key lemma of our proof.
\begin{lemma}
\label{eye}
Let $n,g\in\mathbb{N}$, $g\geq3$. Suppose for every $k\in[0,n]$ we can color $\binom{[n]}{k}$ with $q$ colors so that each color class $\mathcal{F}_{k}^{i}$ with $i\in[q]$ satisfies the following
\begin{equation}
\label{are}
\forall\,A\in\binom{[n]}{k-g}\text{ and }B\in\binom{[n]}{k+g}\text{ with }A\subseteq{B},\,\exists\,T\in\binom{[n]}{k}\backslash\mathcal{F}_{k}^{i}:\,A\subseteq{T}\subseteq{B}.
\end{equation}
Then we have $\chi_{\mu}(Q_{n})\leq{gq}$.
\end{lemma}
\begin{proof}[Proof of Lemma \ref{eye}]
First, we show that the property \eqref{are} is equivalent to the property \eqref{era}. It is obvious that \eqref{era} implies \eqref{are}. It suffices to show the other direction. Fix any $A,B\in2^{[n]}$ with $\lvert{A}\rvert+g\leq{k}\leq\rvert{B}\rvert-g$. Since $\lvert{A\cap{B}}\rvert\leq\lvert{A}\rvert\leq{k-g}$ and $\lvert{A\cup{B}}\rvert\geq\lvert{B}\rvert\geq{k+g}$, there are some $A'\in\binom{[n]}{k-g}$ and $B'\in\binom{[n]}{k+g}$ with $A\cap{B}\subseteq{A'}\subseteq{B'}\subseteq{A\cup{B}}$. Then by \eqref{are} there exists $T\in\binom{[n]}{k}\backslash\mathcal{F}_{k}$ such that $A'\subseteq{T}\subseteq{B'}$, which in particular implies that $A\cap{B}\subseteq{T}\subseteq{A}\cup{B}$.

Now for each $k\in[0,n]$, we color the layer $\binom{[n]}{k}$ with $q$ colors such that every color class satisfies the property \eqref{are} and thus the property \eqref{era}. Let $\mathcal{F}_{k}^{i}\subseteq\binom{[n]}{k}$ denote the $i$th color class in the $k$th layer. It holds that
\begin{equation}
V(Q_{n})=\bigcup_{k=0}^{n}\binom{[n]}{k}=\bigcup_{k=0}^{n}\left(\bigcup_{i=1}^{q}\mathcal{F}_{k}^{i}\right)=\bigcup_{i\in[q],\lambda\in[g]}\left(\bigcup_{k\in{I_{\lambda}}}\mathcal{F}_{k}^{i}\right),
\end{equation}
where $I_{\lambda}=\left\{k\in[0,n]:\,k\equiv\lambda\pmod{g}\right\}$. By Lemma \ref{ear} the set $\bigcup_{k\in{I_{\lambda}}}\mathcal{F}_{k}^{i}$ is a mutual-visibility set in $Q_{n}$. So $V(Q_{n})$ can be partitioned into $gq$ mutual-visibility sets, namely, $\chi_{\mu}(Q_{n})\leq{gq}$. 
\end{proof}

Now we are ready to prove the upper bound in Theorem \ref{new}.

\begin{proof}[Proof of the upper bound]
Assume that $n$ is sufficiently large and all logarithms are in base $2$. By Lemma \ref{eye}, to prove the stated upper bound it suffices to show that for $g=\lfloor{\log\log{n}}\rfloor\geq3$ we can color every layer $\binom{[n]}{k}$ with at most $2$ colors so that each color class satisfies \eqref{are}.

For $k\in[0,n]$ with $k<g$ or $k>n-g$, the whole layer $\binom{[n]}{k}$ satisfies \eqref{are}, because there exists no such pair $(A,B)$ with $A\in\binom{[n]}{k-g}$ and $B\in\binom{[n]}{k+g}$. Hence, the layer $\binom{[n]}{k}$ can be colored with only one color.

For $k\in[0,n]$ with $g\leq{k}\leq{n-g}$, we color the layer $\binom{[n]}{k}$ with $2$ colors uniformly at random. Let
\begin{equation}
\mathcal{J}:=\left\{(A,B):\,A\in\binom{[n]}{k-g}\text{ and }B\in\binom{[n]}{k+g}\text{ with }A\subseteq{B}\right\}.
\end{equation}
For each $(A,B)\in\mathcal{J}$, we define the event $\mathcal{E}_{(A,B)}$ that all $T\in\binom{[n]}{k}$ with $A\subseteq{T}\subseteq{B}$ are colored with the same color. It is not hard to see that
\begin{equation}
\mathbb{P}\left(\mathcal{E}_{(A,B)}\right)\leq{2^{1-\binom{2g}{g}}}=:p.
\end{equation}

Observe that a given event $\mathcal{E}_{(A,B)}$ is mutually independent of all the other events except for those $\mathcal{E}_{(A',B')}$, where there exists $T\in\binom{[n]}{k}$ with $A\subseteq{T}\subseteq{B}$ and $A'\subseteq{T}\subseteq{B'}$. We shall count the number of such events $\mathcal{E}_{(A',B')}$, denoted by $d$. First, note that the number of $T\in\binom{[n]}{k}$ with $A\subseteq{T}\subseteq{B}$ is $\binom{2g}{g}$. Moreover, for every $T\in\binom{[n]}{k}$, there are $\binom{k}{g}\binom{n-k}{g}$ pairs $(A',B')\in\mathcal{J}$ such that $A'\subseteq{T}\subseteq{B'}$. So we have
\begin{equation}
d\leq\binom{2g}{g}\binom{k}{g}\binom{n-k}{g}\leq\binom{2g}{g}\binom{n}{2g}<\left(\frac{en}{g}\right)^{2g}.
\end{equation}
Now since
\begin{equation}
\begin{aligned}
ep(d+1)\leq{2^{1-\binom{2g}{g}}}n^{2g}&\leq2^{-2^{2\log\log{n}}/\sqrt{100\log\log{n}}}n^{2\log\log{n}}\\
&\leq2^{-(\log{n})^{2}/\sqrt{100\log\log{n}}+2\log{n}\log\log{n}}\leq1,
\end{aligned}
\end{equation}
by Lov\'{a}sz Local Lemma\cite{erdos1973local} (see also \cite[Theorem 1.5]{spencer1977local}), there is a positive probability that none of the events $\mathcal{E}_{(A,B)}$ with $(A,B)\in\mathcal{J}$ occurs. Namely, there is a coloring of $\binom{[n]}{k}$ with $2$ colors such that both color classes satisfy \eqref{are}, from which the upper bound follows.
\end{proof}

\end{document}